\documentclass{article}
 
\usepackage{amsmath,amssymb}
\usepackage{amsthm}
\usepackage{enumitem}
\usepackage{multicol}
\usepackage{url} 

\usepackage{color}

\usepackage{tikz}
\tikzset{
every node/.style={circle, draw, inner sep=2pt},
every picture/.style={thick}
}
\usetikzlibrary{matrix,arrows,calc}

\newtheorem{theorem}{Theorem}
\newtheorem{lemma}[theorem]{Lemma}
\newtheorem{proposition}[theorem]{Proposition}
\newtheorem{corollary}[theorem]{Corollary}

\theoremstyle{definition}
\newtheorem{definition}[theorem]{Definition}
\newtheorem{observation}[theorem]{Observation}
\newtheorem{remark}[theorem]{Remark}
\newtheorem{example}[theorem]{Example}
\newtheorem{question}[theorem]{Question}

\newenvironment{thm}{\begin{theorem}}{\end{theorem}}
\newenvironment{lem}{\begin{lemma}}{\end{lemma}}

\newenvironment{cor}{\begin{corollary}}{\end{corollary}}

\def \mr {\operatorname{mr}}
\def \rank {\operatorname{rank}}

\def \S {\mathcal{S}}

\newcommand{\mz}{\operatorname{mz}}

\title{Graph classes for critical ideals, minimum rank and zero forcing number}
\author{Carlos~A.~Alfaro \thanks{Banco de M\'exico, Mexico City, Mexico (alfaromontufar@gmail.com, carlos.alfaro@banxico.org.mx).}
}

\begin{document}

\maketitle

\begin{abstract}
Recently, there have been found new relations between the zero forcing number and the minimum rank of a graph  with the algebraic co-rank.
We continue on this direction by giving a characterization of the graphs with real algebraic co-rank at most 2.
This implies that for any graph with at most minimum rank at most 3, its minimum rank is bounded from above by its real algebraic co-rank.
\end{abstract}



\section{Introduction}
Given a graph $G$ and a set of indeterminates $X_G=\{x_u \, : \, u\in V(G)\}$,
the {\it generalized Laplacian matrix} $L(G,X_G)$ of $G$ is the matrix whose $uv$-entry is given by
\[
L(G,X_G)_{uv}=\begin{cases}
x_u& \text{ if } u=v,\\
-m_{uv}& \text{ otherwise},
\end{cases}
\]
where $m_{uv}$ is the number of the edges between vertices $u$ and $v$.
Moreover, if $\mathcal{R}[X_G]$ is the polynomial ring over a commutative ring $\mathcal{R}$ with unity in the variables $X_G$, then the {\it critical ideals} of $G$ are the determinantal ideals given by
\[
I^{\mathcal{R}}_i(G,X_G)=\langle {\rm minors}_i(L(G,X_G))\rangle\subseteq \mathcal{R}[X_G] \text{ for all } 1\leq i\leq n,
\]
where $n$ is the number of vertices of $G$ and ${\rm minors}_i(L(G,X_G))$ is the set of the determinants of the $i\times i$ submatrices of $L(G,X_G)$.

An ideal is said to be {\it trivial} if it is equal to $\langle1\rangle$ ($=\mathcal{R}[X]$).
The {\it algebraic co-rank} $\gamma_\mathcal{R}(G)$ of $G$ is the maximum integer $i$ for which $I^{\mathcal{R}}_i(G,X_G)$ is trivial.
For simplicity, we might refer to the {\it real algebraic co-rank} to $\gamma_\mathbb{R}(G)$. 
Note that $I^{\mathcal{R}}_n(G,X_G)=\langle \det L(G,X_G)\rangle$ is always non-trivial, and if $d_G$ denote the degree vector, then $I^{\mathcal{R}}_n(G,d_G)=\langle 0\rangle$.

Critical ideals were defined in \cite{corrval} and some interesting properties were pointed out there.
For instance, it was proven that if $H$ is an induced subgraph of $G$, then $I^{\mathcal{R}}_i(H,X_H)\subseteq I^{\mathcal{R}}_i(G,X_G)$  for all $i\leq |V(H)|$.
Thus $\gamma_\mathcal{R}(H)\leq \gamma_\mathcal{R}(G)$.
Initinally, critical ideals were defined as a generalization of the critical group, {\it a.k.a.}\/ sandpile group, see \cite{alfacorrval, alfaval, corrval}.
In \cite{alfaval2,merino} can be found an account of the main results on sandpile group.
Further, it is also a generalization of several other algebraic objects like Smith group or characteristic polynomials of the adjacency and Laplacian matrices, see \cite[Section 4]{alfavalvaz} and \cite[Section 3.3]{corrval}.
In \cite{alflin}, there were explered its relation with the zero forcing number and the minimum rank.
We continue on this direction.
For this, we recall these well-known concepts.

The \emph{zero forcing game} is a color-change game where vertices can be blue or white.
At the beginning, the player can pick a set of vertices $B$ and color them blue while others remain white.
The goal is to color all vertices blue through repeated applications of the \emph{color change rule}: If $x$ is a blue vertex and $y$ is the only white neighbor of $x$, then $y$ turns blue, denoted as $x\rightarrow y$.
An initial set of blue vertices $B$ is called a \emph{zero forcing set} if starting with $B$ one can make all vertices blue.
The \emph{zero forcing number} $Z(G)$ is the minimum cardinality of a zero forcing set.
The \emph{chronological list} of a zero forcing game records the forces $x_i\rightarrow y_i$ in the order of performance.
In the following, $\mz(G)=|V(G)|-Z(G)$.

For a graph $G$ on $n$ vertices, the family $\S_\mathcal{R}(G)$ collects all $n\times n$ symmetric matrices with entries in the ring $\mathcal{R}$, whose $i,j$-entry ($i\neq j$) is nonzero whenever $i$ is adjacent to $j$ and zero otherwise.
Note that the diagonal entries can be any element in the ring $\mathcal{R}$.  
The \emph{minimum rank} $\mr_\mathcal{R}(G)$ of $G$ is the smallest possible rank among matrices in $\S_\mathcal{R}(G)$.
Here we follow \cite[Definition 1]{rankring} and define the rank of a matrix over a commutative ring with unity as the largest $k$ such that there is a nonzero $k\times k$ minor that is not a zero divisor.  In the case of $\mathcal{R}=\mathbb{Z}$, the rank over $\mathbb{Z}$ is the same as the rank over $\mathbb{R}$.

In \cite{AIMZmr}, it was proved that $\mz(G)\leq\mr_\mathcal{R}(G)$ for any field $\mathcal{R}$.
And in \cite{alflin}, it was proved that $\mz(G)\leq\gamma_\mathcal{R}(G)$ for any commutative ring $\mathcal{R}$ with unity.
However, the relation between $\mr_\mathcal{R}(G)$ and $\gamma_\mathcal{R'}(G)$ depends on the rings $\mathcal{R}$ and $\mathcal{R}'$.

Let $I\subseteq \mathcal{R}[X]$ be an ideal in $\mathcal{R}[X]$.
The \emph{variety} of $I$ is defined as
\[
V_\mathcal{R}(I)=\left\{ {\bf a}\in \mathcal{R}^n : f({\bf a}) = 0 \text{ for all } f\in I \right\}.
\]
That is, $V_\mathcal{R}(I)$ is the set of common roots between polynomials in $I$.
We have that
\[
\langle 1\rangle \supseteq I^{\mathcal{R}}_1(G,X_G) \supseteq \cdots \supseteq I^{\mathcal{R}}_n(G,X_G) \supseteq \langle 0\rangle.
\]
Thus
\[
\emptyset=V_\mathcal{R}(\langle 1\rangle) \subseteq V_\mathcal{R}(I^{\mathcal{R}}_1(G,X_G)) \subseteq \cdots \subseteq V_\mathcal{R}(I^{\mathcal{R}}_n(G,X_G)) \subseteq V_\mathcal{R}(\langle 0\rangle)=\mathcal{R}^n.
\]
If $I^\mathcal{R}_k(G,X_G)$ is trivial, then, for all ${\bf a}\in \mathcal{R}^n$, there are $k$-minors of $L(G,{\bf a})$ which are different of 0, and $\rank(L(G,{\bf a}))\geq k$.
However, it does not imply that $\mr_\mathcal{R}(G)\geq \gamma_\mathcal{R}(G)$, since matrices in $\S_\mathcal{R}(G)$ do not necessarily have only $0$ and $-1$ on the off-diagonal entries.
However, if $V_\mathcal{R}(I^{\mathcal{R}}_k(G,X_G))\neq\emptyset$ for some $k$, then there exists ${\bf a}\in\mathcal{R}$ such that, for all $t \geq k$, $I^{\mathcal{R}}_{t}(G,{\bf a})=\langle 0\rangle$; that is, all $t$-minors of $L(G,{\bf a})$ are equal to $0$.
Therefore, $\mr_\mathcal{R}(G)\leq k-1$.
In particular, if $V_\mathcal{R}\left(I^{\mathcal{R}}_{\gamma_\mathcal{R}(G)+1}(G,X_G)\right)$ is not empty, then $\mr_\mathcal{R}(G)\leq \gamma_\mathcal{R}(G)$.
Therefore, as noted in \cite{alflin}, it follows by the Weak Nullstellensatz that if $\mathcal{R}$ is an algebraically closed field, then $\mr_\mathcal{R}(G)\leq \gamma_\mathcal{R}(G)$.
That is not the case for the integers, there exist graphs for which $\mr_\mathbb{Z}(G)> \gamma_\mathbb{Z}(G)$.
For the field of real numbers, it was conjectured \cite{alflin} that $\mr_\mathbb{R}(G)\leq\gamma_\mathbb{R}(G)$.
Trying to sheed some light on this conjecture, it was proved in \cite{alflin} that if $G$ is a connected graph such that $\mr_{\mathbb{R}}(G)\leq 2$, then $\mr_{\mathbb{R}}(G)\leq\gamma_{\mathbb{R}}(G)$.

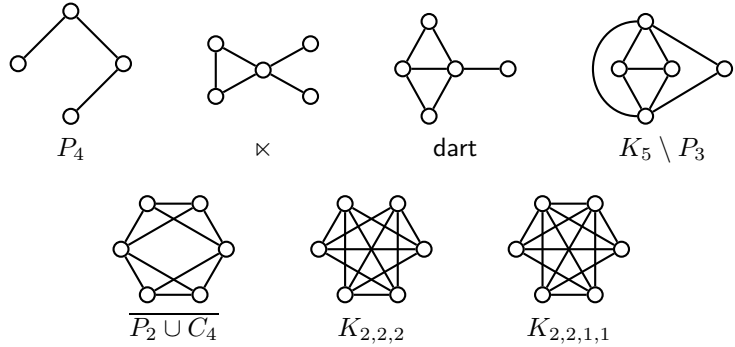
\begin{figure}[h]
\begin{center}
\begin{tabular}{c@{\extracolsep{10mm}}c@{\extracolsep{10mm}}c@{\extracolsep{10mm}}c@{\extracolsep{10mm}}c}
\begin{tikzpicture}[rotate=-90,scale=.7]
	\tikzstyle{every node}=[minimum width=0pt, inner sep=2pt, circle]
\draw (180:1) node (v1) [draw] {};
		\draw (270:1) node (v3) [draw] {};
		\draw (360:1) node (v2) [draw] {};
		\draw (450:1) node (v4) [draw] {};
		\draw (v1) -- (v3);
		\draw (v1) -- (v4);
		\draw (v2) -- (v4);
	\end{tikzpicture}
&
	\begin{tikzpicture}[rotate=-90,scale=.7]
	\tikzstyle{every node}=[minimum width=0pt, inner sep=2pt, circle]
	\draw (-.5,-.9) node (v1) [draw] {};
	\draw (.5,-.9) node (v2) [draw] {};
	\draw (0,0) node (v3) [draw] {};
	\draw (-.5,.9) node (v4) [draw] {};
	\draw (.5,.9) node (v5) [draw] {};
	\draw (v1) -- (v2);
	\draw (v1) -- (v3);
	\draw (v2) -- (v3);
	\draw (v3) -- (v4);
	\draw (v3) -- (v5);
    \node[draw=none] at (0.9,0) {};
	\end{tikzpicture}
&
	\begin{tikzpicture}[scale=.7]
	\tikzstyle{every node}=[minimum width=0pt, inner sep=2pt, circle]
	\draw (-.5,0) node (v2) [draw] {};
	\draw (0,-.9) node (v1) [draw] {};
	\draw (.5,0) node (v3) [draw] {};
	\draw (1.5,0) node (v5) [draw] {};
	\draw (0,.9) node (v4) [draw] {};
	\draw (v1) -- (v2);
	\draw (v1) -- (v3);
	\draw (v2) -- (v3);
	\draw (v2) -- (v4);
	\draw (v3) -- (v4);
	\draw (v3) -- (v5);
	\end{tikzpicture}
    &
	\begin{tikzpicture}[scale=.7]
	\tikzstyle{every node}=[minimum width=0pt, inner sep=2pt, circle]
	\draw (-.5,0) node (v2) [draw] {};
	\draw (0,-.9) node (v1) [draw] {};
	\draw (.5,0) node (v3) [draw] {};
	\draw (1.5,0) node (v5) [draw] {};
	\draw (0,.9) node (v4) [draw] {};
	\draw (v1) -- (v2);
	\draw (v1) -- (v3);
	\draw (v2) -- (v3);
	\draw (v2) -- (v4);
	\draw (v3) -- (v4);
	\draw (v4) -- (v5);
    \draw (v1) -- (v5);
    \draw (v4) to [out=180,in=90] ($(v2)+(-0.5,0)$) to [out=-90, in=180 ] (v1);
	\end{tikzpicture}
\\
$P_4$
&
$\ltimes$
&
{\sf dart}
&
$K_5\setminus{P_3}$ \\
\end{tabular}
\end{center}

\begin{center}
\begin{tabular}{c@{\extracolsep{10mm}}c@{\extracolsep{10mm}}c@{\extracolsep{10mm}}c@{\extracolsep{10mm}}c}
\begin{tikzpicture}[scale=.7]
	\tikzstyle{every node}=[minimum width=0pt, inner sep=2pt, circle]
	\draw (180:1) node (v6) [draw] {};
		\draw (240:1) node (v1) [draw] {};
		\draw (300:1) node (v3) [draw] {};
		\draw (360:1) node (v5) [draw] {};
		\draw (420:1) node (v4) [draw] {};
		\draw (480:1) node (v2) [draw] {};
		\draw (v1) -- (v3);
		\draw (v1) -- (v5);
		\draw (v1) -- (v6);
		\draw (v2) -- (v4);
		\draw (v2) -- (v5);
		\draw (v2) -- (v6);
		\draw (v3) -- (v5);
		\draw (v3) -- (v6);
		\draw (v4) -- (v5);
		\draw (v4) -- (v6);
	\end{tikzpicture}
&
	\begin{tikzpicture}[scale=.7]
	\tikzstyle{every node}=[minimum width=0pt, inner sep=2pt, circle]
	\draw (180:1) node (v1) [draw] {};
		\draw (240:1) node (v2) [draw] {};
		\draw (300:1) node (v3) [draw] {};
		\draw (360:1) node (v4) [draw] {};
		\draw (420:1) node (v5) [draw] {};
		\draw (480:1) node (v6) [draw] {};
		\draw (v1) -- (v3);
		\draw (v1) -- (v4);
		\draw (v1) -- (v5);
		\draw (v1) -- (v6);
		\draw (v2) -- (v3);
		\draw (v2) -- (v4);
		\draw (v2) -- (v5);
		\draw (v2) -- (v6);
		\draw (v3) -- (v5);
		\draw (v3) -- (v6);
		\draw (v4) -- (v5);
		\draw (v4) -- (v6);
	\end{tikzpicture}
    &
	\begin{tikzpicture}[scale=.7]
	\tikzstyle{every node}=[minimum width=0pt, inner sep=2pt, circle]
	\draw (180:1) node (v1) [draw] {};
		\draw (240:1) node (v2) [draw] {};
		\draw (300:1) node (v3) [draw] {};
		\draw (360:1) node (v4) [draw] {};
		\draw (420:1) node (v5) [draw] {};
		\draw (480:1) node (v6) [draw] {};
		\draw (v1) -- (v3);
		\draw (v1) -- (v4);
		\draw (v1) -- (v5);
		\draw (v1) -- (v6);
		\draw (v2) -- (v3);
		\draw (v2) -- (v4);
		\draw (v2) -- (v5);
		\draw (v2) -- (v6);
		\draw (v3) -- (v5);
		\draw (v3) -- (v6);
		\draw (v4) -- (v5);
		\draw (v4) -- (v6);
		\draw (v5) -- (v6);
	\end{tikzpicture}
\\
$\overline{P_2\cup C_4}$
&
$K_{2,2,2}$
&
$K_{2,2,1,1}$
\end{tabular}
\end{center}
\caption{Forbidden graphs for $\Gamma_{\leq2}^{\mathbb{R}}$.}
\label{fig2}
\end{figure}

Let $\Gamma^\mathcal{R}_{\leq i}=\{G\, :\, G \text{ is a simple connected graph with } \gamma_\mathcal{R}(G)\leq i\}$.
Our aim is to give a characterization of $\Gamma^\mathbb{R}_{\leq 2}$.
Given a family of graphs $\mathfrak{F}$, a graph $G$ is called $\mathfrak{F}$-{\it free} if no induced subgraph of $G$ is isomorphic to a member of $\mathfrak{F}$.
We will characterize $\Gamma^\mathbb{R}_{\leq 2}$ as the $\{P_4,\ltimes,{\sf dart},K_5\setminus{P_3},\overline{P_2\cup C_4},K_{2,2,2},K_{2,2,1,1}\}$-free graphs.
Since $\mr_{\mathbb{R}}(G)\leq 2$ if and only if $G$ is \{$P_4$,$K_{3,3,3}$,$\ltimes$,{\sf dart}\}-free, then we have that any graph $G\in\Gamma^\mathbb{R}_{\leq 2}$ has $\mr_{\mathbb{R}}(G)\leq 2$.
Thus, if $G$ is a connected graph such that $\mr_{\mathbb{R}}(G)= 3$, then $\gamma(G)\geq3$.
Implying that if $G$ is a connected graph such that $\mr_{\mathbb{R}}(G)\leq 3$, then $\mr_{\mathbb{R}}(G)\leq\gamma_{\mathbb{R}}(G)$.

The paper is organized as follows.
In Section~\ref{section:graphclasses}, we will give an overview of the main classifications that have been obtained for graphs with bounded $\mz$, $\mr$ and $\gamma$.
We will give a characterization of the $\{P_4, \ltimes, {\sf dart}, K_5\setminus{P_3}, \overline{P_2\cup C_4}, K_{2,2,2}, K_{2,2,1,1}\}$-free graphs.
In Section~\ref{section:blowupgraphs}, we will recall a method to compute the algebraic co-rank of blowup graphs.
And we will use it to prove that in fact the given characterization is of the graphs with real minimum rank at most 2.

\section{Graph classes for bounded $\mz$, $\mr$ and $\gamma$}\label{section:graphclasses}

It is known that algebraic co-rank, minimum rank and $\mz$ are monotone on induced subgraphs, that is, if $H$ is an induced subgraph of $G$, then $\gamma_\mathcal{R}(H)\leq \gamma_\mathcal{R}(G)$, $\mr_\mathcal{R}(H)\leq \mr_\mathcal{R}(G)$ and $\mz(H)\leq \mz(G)$.
Then, it is natural to ask for classifications of graphs where these parameters are bounded from above.

Since $\mz(G)\leq \gamma_{\mathcal{R}}(G)$ and $\mz(G)\leq \mr_{\mathcal{R}}(G)$, then the family of graphs with $\gamma_{\mathcal{R}}(G)\leq k$ or $\mr_{\mathcal{R}}(G)\leq k$ are contained in the family of graphs with $\mz(G)\leq k$.
However, the relation between the families of graphs with $\gamma_{\mathcal{R}}(G)\leq k$ and $\mr_{\mathcal{R}}(G)\leq k$ is still not clear.

In previous works, it was noticed in \cite{alfaval,BHL04} that among all connected graphs, the complete graphs are the only graphs whose minimum rank, algebraic co-rank and $\mz$ are equal to 1.
Also, in \cite[Theorem 16]{BHL04} it was proved that for any connected graph $G$, $\mz(G)\leq2$ if and only if $G$ is $\{P_4,\ltimes,{\sf dart}\}$-free.
In \cite{BHL04,BHL05}, there are classifications of graphs whose minimum rank is at most 2 depending on the base field.
In particular for the field of real numbers, we have the following result, where $G+H$ denote the {\it disjoint union} of the graphs of $G$ and $H$, and $G\vee H$ denote the {\it join} of $G$ and $H$.
\begin{thm}\cite{BF07,BHL04}
    Let $G$ be a connected graph.
    Then, the following are equivalent:
    \begin{enumerate}
        \item $\mr(G)\leq 2$,
        \item $G$ is \{$P_4$,$K_{3,3,3}$,$\ltimes$,{\sf dart}\}-free,
        \item $G=\bigvee_{i=1}^r G_i$, $r>1$, where either
            \begin{enumerate}
                \item $G_i=K_{m_i}+ K_{n_i}$ for suitable $m_i\geq 1$, $n_i\geq0$, or
                \item $G_i=\overline{K_{m_i}}$ for a suitable $m_i\geq3$;
            \end{enumerate}
            and option $(b)$ occours at most twice.
    \end{enumerate}
\end{thm}
On the other hand, we have that if $\mathcal{R}'$ is a subring of $\mathcal{R}$, then $\gamma_\mathcal{R'}(G)\leq\gamma_\mathcal{R}(G)$.
From which follows $\Gamma^{\mathbb{R}}_{\leq k}\subseteq\Gamma^{\mathbb{Z}}_{\leq k}$.
In this sense, in \cite{alfaval} the connected graphs with $\gamma_{\mathbb{Z}}(G)\leq 2$ were classified.

\begin{thm}\label{thm:chacterizationZ2}
    Let $G$ be a connected graph.
    Then, the following are equivalent:
    \begin{enumerate}
        \item $G\in\Gamma^{\mathbb{Z}}_{\leq 2}$,
        \item \{$P_4$, $K_{2,2,1,1}$, $K_5\setminus{P_3}$, $\ltimes$, ${\sf dart}$\}-free graphs,
        \item $G$ is isomorphic to $K_{n_1,n_2,n_3}$ or to $\overline{K_{n_1}}\vee(K_{n_2}+K_{n_3})$.
    \end{enumerate}
\end{thm}
Few is known for graphs with minimum rank and algebraic co-rank at most 3.
In \cite{alfaval1,BGL09}, there were obtained only partial results for the minimum rank and algebraic co-rank at most 3.
And the problem still seems to be far to be completely understod.
And in \cite{alflin,alfavalvaz}, there were characterized the digraphs whose minimum rank, algebraic co-rank and $\mz$ are equal to 1.

A graph $G$ is \textit{forbidden} for $\Gamma^{\mathcal{R}}_{\leq k}$ when $\gamma_\mathcal{R}(G)\geq k+1$.
Let ${\bf Forb}(\Gamma^{\mathcal{R}}_{\leq k})$ be the set of minimal (under induced subgraphs property) forbidden graphs for $\Gamma^\mathcal{R}_{\leq k}$.
A graph $G$ is $\gamma_\mathcal{R}$-{\it critical} if $\gamma_\mathcal{R}(G-v)<\gamma_\mathcal{R}(G)$ for each $v\in V(G)$.
Then, $G\in {\bf Forb}(\Gamma^{\mathcal{R}}_{\leq k})$ if and only if $G$ is $\gamma_\mathcal{R}$-critical such that $\gamma_\mathcal{R}(G-v)\leq k$ and $k<\gamma_\mathcal{R}(G)$ for each $v\in V(G)$.
Therefore $G\in\Gamma^{\mathcal{R}}_{\leq k}$ if and only if $G$ is ${\bf Forb}(\Gamma^{\mathcal{R}}_{\leq k})$-free.
Thus, characterizing ${\bf Forb}(\Gamma^{\mathcal{R}}_{\leq k})$ leads to a characterization of $\Gamma^{\mathcal{R}}_{\leq k}$.

Since $\gamma_\mathbb{Z}(G)\leq\gamma_\mathbb{R}(G)$ for any graph $G$, then we have that $P_4$, $K_{2,2,1,1}$ and $K_5\setminus{P_3}$ are forbidden graphs for $\Gamma^\mathbb{R}_{\leq 2}$.
In fact we have the following.

\begin{lem}\label{lemma:free}
    The graphs $P_4,\ltimes,{\sf dart},K_5\setminus{P_3},\overline{P_2\cup C_4},K_{2,2,2}$ and $K_{2,2,1,1}$ are in ${\bf Forb}(\Gamma^{\mathbb{R}}_{\leq 3})$.
\end{lem}

This can be verified by using a Computer Algebra System like {\it Macaulay2}.
More precisely, it can be proved that these graphs are $\gamma_\mathbb{R}$-{\it critical} and their real algebric co-rank is 3.
At this moment it does not imply that these graphs are all the graphs in ${\bf Forb}(\Gamma^{\mathbb{R}}_{\leq 3})$.

\begin{example}
    Let us consider the Gr\"obner bases of the third critical ideal on $\mathbb{Z}$ of $K_{2,2,2}$:
    \[
    I_3^\mathbb{Z}\left(K_{2,2,2},X_{K_{2,2,2}}\right)=\langle x_1, x_2, x_3, x_4, x_5, x_6, 2 \rangle.
    \]
    When we consider this ideal over the real numbers, it becomes trivial.
    Similarly, the Gr\"obner bases of the third critical ideal on $\mathbb{Z}$ of $\overline{P_2\cup C_4}$ is not trivial:
    \[
    I_3^\mathbb{Z}\left(\overline{P_2\cup C_4},X_{\overline{P_2\cup C_4}}\right)=\langle x_1 + 1, x_2 + 1, x_3 + 1, x_4 + 1, x_5, x_6, 2 \rangle,
    \]
    where $v_5$ and $v_6$ are the vertices of degree 4.
    And again, when we consider this ideal over the real numbers, it becomes trivial.
    This is an interesting behaviour that does not happen on the rest of graphs in ${\bf Forb}(\Gamma^{\mathbb{R}}_{\leq 2})$.
\end{example}

We start from the characterization of ${\bf Forb}(\Gamma^{\mathbb{Z}}_{\leq 2})$, and, additionally, the induced subgraphs $K_{2,2,2}$ and $\overline{P_2\cup C_4}$ will be removed.

\begin{lem}\label{lemma:forb}
Let $G$ be a connected graph.
    Then, $G$ is $\{P_4,\ltimes,{\sf dart}, K_5\setminus{P_3}, \overline{P_2\cup C_4}, K_{2,2,2},K_{2,2,1,1}\}$-free if and only if $G$ is isomorphic to an induced subgraph of one of the following graphs: $K_{1,n_1,n_2}$, $\overline{K_{1}}\vee(K_{n_2}+K_{n_3})$ or $\overline{K_{n_1}}\vee(K_{1}+K_{n_3})$.
\end{lem}

\begin{proof}
    Let $G$ be $\{P_4,\ltimes,{\sf dart}, K_5\setminus{P_3}, K_{2,2,1,1}\}$-free.
    By Theorem~\ref{thm:chacterizationZ2}, we have two cases, either $G$ is isomorphic to $K_{n_1,n_2,n_3}$ or to $\overline{K_{n_1}}\vee(K_{n_2}+K_{n_3})$.
    In the first case, since $K_{2,2,2}$ is forbidden for $G$,  we have that at least one of the $n_1,n_2,n_3$ must be at most 1.
    In the second case, we can observe that $\overline{P_2\cup C_4}$ can be regarded as $\overline{K_{2}}\vee(K_{2}+K_{2})$.
    From which follows that either $n_1\leq 1$ or at least one of $n_2$ and $n_3$ is at most 1.
    The other direction follows since $\overline{P_2\cup C_4}$ is not and induced subgraph of $K_{1,n_2,n_3}$, and $K_{2,2,2}$ is not an induced subgraph of $\overline{K_{1}}\vee(K_{n_2}+K_{n_3})$ nor $\overline{K_{n_1}}\vee(K_{1}+K_{n_3})$.
\end{proof}

It remains to prove that $P_4,\ltimes,{\sf dart},K_5\setminus{P_3},\overline{P_2\cup C_4},K_{2,2,2}$ and $K_{2,2,1,1}$ are in fact all the graphs in ${\bf Forb}(\Gamma^{\mathbb{R}}_{\leq k})$.
This can be done by computing the algebric co-rank of the graphs in $K_{1,n_1,n_2}$, $\overline{K_{1}}\vee(K_{n_2}+K_{n_3})$ and $\overline{K_{n_1}}\vee(K_{1}+K_{n_3})$, and checking that any graph  $G$ in these families has $\gamma_\mathbb{R}(G)\leq2$.
That will be done in the following section.

\section{Blowup graphs}\label{section:blowupgraphs}

Given a graph $G=(V,E)$ and a vector ${\bf d}\in {\mathbb Z}^V$, the graph $G^{\bf d}$ is constructed as follows.
For each vertex $u\in V$, associate a new vertex set $V_u$, where $V_u$ is a clique of cardinality $-{\bf d}_u$ when ${\bf d}_u$ is negative, and $V_u$ is a stable set of cardinality ${\bf d}_u$ if ${\bf d}_u$ when positive.
Each vertex in $V_u$ is adjacent with each vertex in $V_v$ if and only if $u$ and $v$ are adjacent in $G$.
Then the graph $G$ is called the {\it underlying graph} of $G^{\bf d}$.

In general, the computation of the Gr\"{o}bner bases of the critical ideals is more than complicated.
However, we will use a method, developed in \cite{alfacorrval}, to decide, for $i\leq|V(G)|$, whether the $i$-{\it th} critical ideal of $G^{\bf d}$ is trivial or not.

For ${\bf d}\in {\mathbb Z}^{V}$, we define $\phi({\bf d})$ as follows:
\[
\phi({\bf d})_v =
\begin{cases}
0 & \text{ if }{\bf d}_v>1,\\
-1 & \text{ if }{\bf d}_v<-1,\\
x_v & \text{ otherwise }.
\end{cases}
\]

\begin{thm}\cite[Theorem 2.7]{alfacorrval}\label{Theo:trivialcriticalidealsiff}
Let $n\geq 2$ and $G=(V,E)$ be a graph with $n$ vertices.
For $1\leq j\leq n$ and ${\bf d}\in \mathbb{Z}^{V}$,  the critical ideal $I^\mathcal{R}_j(G^{\bf d},X_{G^{\bf d}})$ is trivial if and only if the evaluation of $I^\mathcal{R}_j(G,X_G)$ at $X_G=\phi({\bf d})$ is trivial.
\end{thm}

Therefore, verifying whether a family of graphs have algebraic co-rank at most $i$ becomes in an evaluation of the $i$-{\it th} critical ideal of the underlying graph of the family.
It might be possible that such a family might be described by an infinite number of underlying graphs.

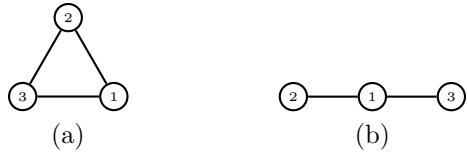
\begin{figure}[h!]
    \begin{center}
    \begin{tabular}{c@{\extracolsep{2cm}}c}
    \begin{tikzpicture}[scale=.7,thick]
	\tikzstyle{every node}=[minimum width=0pt, inner sep=2pt, circle]
	\draw (-30:1) node[draw] (0) {\tiny 1};
	\draw (90:1) node[draw] (1) {\tiny 2};
	\draw (210:1) node[draw] (2) {\tiny 3};
	\draw  (0) edge (1);
    \draw  (1) edge (2);
	\draw  (0) edge (2);
    \end{tikzpicture}
    &
        \begin{tikzpicture}[scale=.7,thick]
	\tikzstyle{every node}=[minimum width=0pt, inner sep=2pt, circle]
	\draw (-1.5,0) node[draw] (0) {\tiny 2};
	\draw (0,0) node[draw] (2) {\tiny 1};
	\draw (1.5,0) node[draw] (4) {\tiny 3};
	\draw  (0) edge (2);
    \draw  (2) edge (4);
    \end{tikzpicture}
        \\
    (a) & (b)
    \end{tabular}
    \end{center}
    \caption{(a) The underlying graph for $K_{1,n_1,n_2}$, and (b) the underlying graph for $\overline{K_{n_1}}\vee(K_{n_2}+K_{n_3})$.}
    \label{fig:blowupgraph}
\end{figure}

\begin{lem}\label{lemma:gamma1}
    Let $G$ be an induced subgraph of $K_{1,n_1,n_2}$, then $\gamma_\mathbb{R}(G)\leq 2$.
\end{lem}
\begin{proof}
    The underlying graph $H$ of $K_{1,n_1,n_2}$ is shown in Figure~\ref{fig:blowupgraph}.a.
    We have that
    \[I^{\mathbb{R}}_3(H,X_H)=\langle x_1x_2x_3 - x_1 - x_2 - x_3 - 2\rangle\]
    Let ${\bf d}=(0,-n_1,-n_2)$, and thus $\phi({\bf d})=(x_1,-1,-1)$.
    By evaluating the third critical ideal at $\phi({\bf d})$, we have $I_3(H,\phi({\bf d}))=\langle 0\rangle$.
    By Theorem~\ref{Theo:trivialcriticalidealsiff}, $I^\mathbb{R}_3(H^{\bf d},X_{H^{\bf d}})$ is not trivial and  $\gamma_\mathbb{R}(K_{1,n_1,n_2})\leq 2$.
\end{proof}

\begin{lem}\label{lemma:gamma2}
    Let $G$ be an induced subgraph of $\overline{K_{1}}\vee(K_{n_2}+K_{n_3})$ or $\overline{K_{n_1}}\vee(K_{1}+K_{n_3})$, then $\gamma_\mathbb{R}(G)\leq 2$.
\end{lem}
\begin{proof}
    The underlying graph $H$ of $\overline{K_{n_1}}\vee(K_{n_2}+K_{n_3})$ is shown in Figure~\ref{fig:blowupgraph}.b.
    We have that
    \[I^{\mathbb{R}}_3(H,X_H)=\langle x_1x_2x_3 - x_2 - x_3\rangle\]
    Let ${\bf d}_1=(n_1,0,-n_3)$, and thus $\phi({\bf d}_1)=(0,x_2,-1)$.
    By evaluating the third critical ideal at $\phi({\bf d}_1)$, we have $I_3(H,\phi({\bf d}_1))=\langle -x_2+1\rangle$.
    By Theorem~\ref{Theo:trivialcriticalidealsiff}, $\gamma_\mathbb{R}\left(\overline{K_{n_1}}\vee(K_{1}+K_{n_3})\right)\leq 2$.
    Let ${\bf d}_2=(0,-n_2,-n_3)$, and thus $\phi({\bf d}_2)=(x_1,-1,-1)$.
    By evaluating the third critical ideal at $\phi({\bf d}_2)$, we have $I_3(H,\phi({\bf d}_2))=\langle x_1+2\rangle$.
    By Theorem~\ref{Theo:trivialcriticalidealsiff}, $\gamma_\mathbb{R}\left(\overline{K_{1}}\vee(K_{n_2}+K_{n_3})\right)\leq 2$.
\end{proof}

Lemmas~\ref{lemma:free}, \ref{lemma:forb}, \ref{lemma:gamma1} and \ref{lemma:gamma2} imply our main result.

\begin{thm}
    Let $G$ be a connected graph.
    Then, the following are equivalent.
    \begin{enumerate}
        \item $G\in\Gamma^{\mathbb{R}}_{\leq 2}$,
        \item $G$ is $\{P_4,\ltimes,{\sf dart}, K_5\setminus{P_3}, \overline{P_2\cup C_4}, K_{2,2,2},K_{2,2,1,1}\}$-free,
        \item $G$ is isomorphic to an induced subgraph of one of the following graphs: $K_{1,n_1,n_2}$, $\overline{K_{1}}\vee(K_{n_2}+K_{n_3})$ or $\overline{K_{n_1}}\vee(K_{1}+K_{n_3})$.
    \end{enumerate}
\end{thm}

The fact that $K_{2,2,2}$ is an induced subgraph of $K_{3,3,3}$ implies that if $G\in\Gamma^{\mathbb{R}}_{\leq 2}$, then $\mr(G)\leq 2$.
Therefore, if $G$ is a connected graph such that $\mr(G)= 3$, then $\gamma_{\mathbb{R}}(G)\geq3$.
Which implies the following result.

\begin{cor}
    If $G$ is a connected graph such that $\mr(G)= 3$, then $\mr(G)\leq\gamma_{\mathbb{R}}(G)$.
\end{cor}

\section*{Acknowledgments}
This research was partially supported by SNI and CONACyT.


\end{document}